\documentclass[leqno,10pt]{amsart}

\usepackage[english]{babel}
\usepackage[T1]{fontenc}

\usepackage{amsmath,amssymb,amsthm}
\usepackage{mathtools}  
\usepackage[margin=1.25in]{geometry}

\usepackage{graphics,tikz,caption,subcaption}

\usepackage{mathrsfs}
\usepackage{stmaryrd}

\usepackage[colorlinks	=	true, linkcolor	=	red, urlcolor	=	red, citecolor	=	red]{hyperref}
\usepackage{bookmark}									
\usepackage{cleveref}

\usepackage{multicol}
\usepackage{todonotes}
\usepackage{enumitem}
\usepackage{verbatim}

\theoremstyle{definition}
\newtheorem{thm}{Theorem}[section]
\newtheorem{exm}[thm]{Example}

\newtheorem{lemm}[thm]{Lemma}

\newtheorem{rem}[thm]{Remark}

\newtheorem{prop}[thm]{Proposition}
\newtheorem{ques}[thm]{Question}

\newtheoremstyle{case}
{3pt}
{3pt}
{}
{}
{\itshape}
{:}
{.5em}
{}

\theoremstyle{case}

\theoremstyle{remark}

\DeclareMathOperator{\diam}{diam}

\allowdisplaybreaks

\newcommand{\apmd}[2][]{															
	\ifthenelse{\equal{#1}{}}%
					{ \operatorname{N}_{#2}	}%
					{ \operatorname{N}_{#1}(#2) 	}}
					
\begin{document}

	\title{Graphs that are not minimal for conformal dimension}
	\author{Matthew Romney}
	\address{Department of Mathematical Sciences\\ Stevens Institute of Technology, Hoboken, NJ 07030}
    \email{mromney@stevens.edu}
	
	\date{}
    \thanks{The author is partially supported by the National Science Foundation under Grant No. DMS-2413156.}

	\maketitle
	
	\begin{abstract}
	    We construct functions $f \colon [0,1] \to [0,1]$ whose graph as a subset of $\mathbb{R}^2$ has Hausdorff dimension greater than any given value $\alpha \in (1,2)$ but conformal dimension $1$. These functions have the property that a positive proportion of level sets have positive codimension-$1$ measure. This result gives a negative answer to a question of Binder--Hakobyan--Li. We also give a function whose graph has Hausdorff dimension $2$ but conformal dimension $1$. The construction is based on the author's previous solution to the inverse absolute continuity problem for quasisymmetric mappings.  
	\end{abstract}

	
	

	\section{Introduction}

    The \textit{conformal dimension} of a metric space is the infimal Hausdorff dimension of its image under all quasisymmetric mappings from that space. The notion was introduced by Pansu in \cite{Pan:89} and has become standard in analysis on metric spaces and neighboring fields. See the monograph of Mackay and Tyson for a detailed introduction to the topic \cite{MT:10}. 
    
    A problem of great interest is to recognize when the conformal dimension of a space is its Hausdorff dimension; such a space is said to be \textit{minimal for conformal dimension}. A theorem of Tyson \cite{Tys:00} states that a compact metric space is minimal for conformal dimension provided it is upper Ahlfors $\alpha$-regular for some $\alpha>1$ and contains a family of curves of positive $\alpha$-modulus (see, for example, \cite[Chapter 7-8]{Hei:01} for the relevant definitions). In particular, given any Ahlfors regular compact set $E \subset \mathbb{R}^n$, the product set $E \times [0,1] \subset \mathbb{R}^{n+1}$ is minimal for conformal dimension. Tyson's theorem was generalized by Hakobyan \cite[Theorem 5.5]{Hak:10}, and further by Binder--Hakobyan--Li \cite[Theorem 3.2]{BHL:23}. In these results, the assumption of a curve family of positive modulus is replaced by a more general condition that the space contains a family of subsets each of conformal dimension at least $1$, and that these subsets support a collection of sufficiently regular measures with positive modulus in the sense of Fuglede \cite{Fug:57}. See the respective papers for a precise statement of these results.



    In \cite{BHL:23}, Binder--Hakobyan--Li succesfully apply their criterion to prove minimality for conformal dimension for so-called \textit{Bedford--McMullen carpets with uniform fibers} and, most notably, the graph of $1$-dimensional Brownian motion (almost surely). They further observe that, by using the class of Bedford--McMullen carpets with uniform fibers, one can construct graphs of continuous functions of the unit interval $[0,1]$ into itself with arbitrary Hausdorff dimension $\alpha \in [1,2]$. Based on these positive results, Binder--Hakobyan--Li raise the question of what conditions guarantee that the graph of a continuous function $f \colon [0,1] \to \mathbb{R}$ is minimal for conformal dimension. They conjecture \cite[Conjecture 7.1]{BHL:23} that this should be true assuming there is a set $A \subset \mathbb{R}$ of positive Lebesgue measure such that for all $a \in A$, the set $\Gamma(f) \cap (\mathbb{R} \times \{a\})$ has Hausdorff dimension $\dim_H(\Gamma(f))-1$. Here, $\Gamma(f)$ denotes the graph of $f$. 
    

    In this paper, we give counterexamples to the previous conjecture. 
    
    \begin{thm} \label{thm:main_a}
        For all $\alpha \in (1,2)$, there is a function $f \colon [0,1] \to [0,1]$ whose graph $\Gamma(f)$ has Hausdorff dimension greater than $\alpha$ but for which there is a quasisymmetric mapping $g \colon \Gamma(f) \to Y$ onto a metric space $Y$ of Hausdorff dimension $1$. This $f$ has the property that $\Gamma(f) \cap (\mathbb{R} \times \{a\})$ has Hausdorff dimension $\dim_H(\Gamma(f))-1$ for every $a \in [0,1]$ in a set of positive Lebesgue measure.
    \end{thm}

    By gluing together functions as constructed in \Cref{thm:main_a}, we can also give an example that is the most extreme possible from the point of view of dimensions:
    
    \begin{thm} \label{thm:main}
        There is a function $f \colon [0,1] \to [0,1]$ whose graph $\Gamma(f)$ has Hausdorff dimension $2$ but for which there is a quasisymmetric mapping $g \colon \Gamma(f) \to Y$ onto a metric space $Y$ of Hausdorff dimension $1$. 
    \end{thm}

    We recall the fundmental definition: a homeomorphism $f \colon X \to Y$ between metric spaces $(X,d_X)$ and $(Y, d_Y)$ is \textit{quasisymmetric} if there exists a homeomorphism $\eta \colon [0, \infty) \to [0, \infty)$ such that for all distinct points $x,y,z \in X$,
    \[\frac{d_Y(f(x),f(y))}{d_Y(f(x),f(z))} \leq \eta\left(\frac{d_X(x,y)}{d_X(x,z)}\right).\]
    Quasisymmetric maps roughly preserve relative distance between triples of points, as dictated by the control function $\eta$. Conformal mappings between planar domains satisfy the quasisymmetry condition locally, and so quasisymmetric maps can be considered a natural generalization of conformal maps to the setting of metric spaces. 
    
    Our construction is based on the author's solution to the inverse absolute continuity problem for quasisymmetric mappings on Euclidean space \cite{Rom:19b}. This problem asks whether a quasisymmetric mapping from $\mathbb{R}^n$ onto some other metric space can map a set of positive Lebesgue $n$-measure onto a set of Hausdorff $n$-measure zero. See \cite[Questions 5,6]{HeiS:97} for a statement. It is shown in \cite{Rom:19b} that this can be done, and in fact the Hausdorff dimension can be lowered to less than an arbitrary positive value. The relevant result for our purpose (specialized to $\mathbb{R}^2$) is the following. 
    
    \begin{thm} \label{thm:qs}
        Fix constants $\delta>0$ and $\beta \in (0,1)$. There is a metric space $X$ and quasisymmetric mapping $g_0 \colon [0,1]^2 \to X$ with the following two properties:
        \begin{itemize}
            \item[(1)] $g_0$ is $L$-Lipschitz for some fixed value $L$.
            \item[(2)] There is a set $E \subset [0,1]^2$ with Lebesgue measure greater than $\beta$ such that the Hausdorff dimension of $g_0(E)$ is at most $\delta$.
        \end{itemize}
    \end{thm}

    The basic idea to prove \Cref{thm:main_a} is to construct a function whose graph oscillates greatly within the set $E$ (which gets mapped onto a set of small dimension) while its graph outside of $E$ is locally rectifiable. \Cref{thm:main_a} is then a consequence of properties (1) and (2) of \Cref{thm:qs}. 
    
    In the proof of \Cref{thm:main_a}, we use two separate constructions superimposed on each other. First is the construction of a quasisymmetric map $g_0$ and corresponding set $E$ as in \Cref{thm:qs} (for fixed values of $\beta$ and $\delta$), which we borrow from \cite{Rom:19b} except with one modification needed for our application here. This is carried out in \Cref{sec:g0}. Next, in \Cref{sec:fn}, we define for each $n \in \mathbb{N}$ a function $f_n \colon [0,1] \to [0,1]$ by taking a standard Bedford--McMullen carpet with uniform fibers of Hausdorff dimension $D_n = 2-1/n$, then modifying it outside the set $E$ to be locally rectifiable there. The main step in this section is showing that $\Gamma(f_n) \cap E$ still has positive Hausdorff $D_n$-measure. In \Cref{sec:proof}, we obtain \Cref{thm:main_a} by taking $f= f_n$ for sufficiently large $n$ and letting $g$ be the restriction of $g_0$ to $\Gamma(f_n)$ and $Y$ be the image of $g$. Finally, we combine rescaled version of the functions $f_n$ and map $g_0$ together to obtain the $f$ and $g$ required in \Cref{thm:main}.


    We conclude this introduction with a few additional remarks. 

       \begin{rem}
        The map $g$ and the metric space $Y$ we use to prove \Cref{thm:main_a} and \Cref{thm:main} do not have any particularly nice structure beyond the requirements of the theorem. One might refine the original problem of Binder--Hakobyan--Li, for example, by requiring $g$ to be the restriction of a quasiconformal homeomorphism of $\mathbb{R}^2$ or $\mathbb{R}^3$. We do not know whether such a map $g$ exists.
        
        One of the main results of the paper \cite{NR:20} states that the map $g_0$ in \Cref{thm:qs} can be taken to be the restriction of a quasiconformal homeomorphism of $\mathbb{R}^3$, but with the weaker conclusion that the Hausdorff $2$-measure of $g_0(E)$ is zero, rather than $g_0(E)$ having smaller Hausdorff dimension than $E$. Because of this weaker conclusion, it does not seem possible to leverage the result in \cite{NR:20} to show that $g$ can be the restriction of a quasiconformal homeomorphism of $\mathbb{R}^3$. 
    \end{rem}
    
    Next, to put the ideas of this paper in perspective, we give the following example, which we learned from Hrant Hakobyan.

   \begin{exm} \label{exm}
       A similar but simpler construction gives functions as in \Cref{thm:main_a} but without the property regarding the Hausdorff dimension of $\Gamma(f) \cap (\mathbb{R} \times \{a\})$. We briefly outline this construction. Consider a sawtooth pattern such as the one in \Cref{fig:pattern1} but with the rectangles in one row deleted. One obtains a totally disconnected self-affine set $F$ by iterating this pattern; this can be carried out so the Hausdorff dimension of $F$ is at least $\alpha$ for any $\alpha \in (1,2)$. The set $F$ has the property that each projection onto a coordinate axis has Hausdorff dimension less than $1$. It follows from a result of Mackay \cite{Mac:11} that $F$ has conformal dimension $0$. In particular, given a value $\varepsilon \in (0,1)$, there is a quasisymmetric map $g_\varepsilon$ from $F$ to a metric space $Y_\varepsilon$ of Hausdorff dimension less than $\varepsilon$. 

        Define a continuous function $f \colon [0,1] \to [0,1]$ by filling in the holes in $F$ with straight line segments. Then $F \subset \Gamma(f)$. One can extend the quasisymmetry $g_\varepsilon$ from the previous paragraph to a quasisymmetric map defined on $\Gamma(f)$, also denoted by $g_\varepsilon$, with each straight line segment getting mapped to a rectifiable curve. Thus $g_\varepsilon(\Gamma(f))$ has Hausdorff dimension $1$, and so $\Gamma(f)$ has conformal dimension $1$. 
        
        However, for almost every $a \in [0,1]$, the level set of $\Gamma(f) \cap (\mathbb{R} \times\{a\})$ has Hausdorff dimension $0$. In addition, this example does not seem to yield \Cref{thm:main}, at least without more careful work in constructing the quasisymmetric map $g_\varepsilon$. Our proof relies on the fact that \Cref{thm:qs} provides a single quasisymmetric mapping defined on the unit square that can be applied to the functions obtained in \Cref{thm:main_a} independently of $\alpha$. 
   \end{exm}

   In addition to the previous example, see \cite[Section 7]{BHL:23} for an example of a non-continuous function whose graph has conformal dimension $0$, based on a theorem of Tukia \cite{Tuk:89} on quasisymmetric distortion of subsets of $\mathbb{R}$ (in fact, \Cref{thm:qs} can be thought of as an analogue of Tukia's theorem for higher dimensions).

        Finally, in light of the results in this paper, it remains an open problem to give sufficient conditions for the graph of a continuous function to be minimal for conformal dimension. We suggest a possibility:

        \begin{ques}
            Let $f \colon [0,1] \to [0,1]$ be a continuous function whose graph has Hausdorff dimension $d$ and is homogeneous in the following sense: there exists a constant $C>1$ such that \[C^{-1}(t-s) \leq \mathcal{H}^d(\Gamma(f|_{[s,t]})) \leq C(t-s)\] for all $0 \leq s < t \leq 1$. Is $\Gamma(f)$ necessarily minimal for conformal dimension?
        \end{ques}

        Here, $\mathcal{H}^d$ denotes $d$-dimensional Hausdorff measure. 

    \subsection*{Acknowledgements}

    I thank Hrant Hakobyan for feedback on a draft of this paper, and in particular for suggesting \Cref{exm}. I also thank Chris Bishop for helpful discussions on the topic.

    \section{Constructing the quasisymmetric map} \label{sec:g0}

    In this section, we explain the construction of a quasisymmetric map $g_0$ satisfying the properties of \Cref{thm:qs}. This construction we give is essentially the same as that found in  \cite{Rom:19b}, although we modify one of the details in order to find the set $E$ more explicitly than in \cite{Rom:19b}. As a further advantage, the set $E$ found here is compact, unlike the set found in \cite{Rom:19b}. It depends on the choice of a sufficiently large odd integer $M \in \mathbb{N}$ and a sufficiently small scaling factor $r>0$. These constants are fixed once for the entire paper. For the argument presented here to work, we may take any $M \geq 78$ and any $r < M^{-3}$.

    Start with the unit square $Q_0 = [0,1]^2$. Divide $Q_0$ into an $M \times M$ grid of subsquares 
    \[Q(i_1, j_1) = \left\{(x,y): \frac{i_1-1}{M} \leq x \leq \frac{i_1}{M}, \frac{j_1-1}{M} \leq y \leq \frac{j_1}{M}\right\}, \]
    indexed by parameters $i_1, j_1 \in \{1, \ldots, M\}$. Denote the collection of these squares by $\mathcal{Q}_1$. Each square $Q(i_1,j_1) \in \mathcal{Q}_1$ is then further divided into an $M \times M$ grid of subsquares $Q(i_1,j_1,i_2,j_2)$ in like manner, and so forth. In this way we obtain for all $m \in \mathbb{N}$ a collection $\mathcal{Q}_m$ of squares \[Q(i_1, j_1, \ldots, i_m, j_m)\] of side length $M^{-m}$, where $i_1,j_1, \ldots, i_m,j_m \in \{1,\ldots, M\}$. We write $I_m$ to denote the multiindex $(i_1, j_1, \ldots, i_m, j_m)$.

    The set of indices for each level are divided into three types:
    \begin{itemize}
        \item $\mathcal{I}_1$ denotes the set of tuples ${(i,j) \in \{1,\ldots, M\}^2}$ such that $i=1$, $i=M$, $j=1$ or $j=M$.
        \item $\mathcal{I}_2$ denotes the set of tuples $(i,j) \in \{1,\ldots, M\}^2$ not in $\mathcal{I}_1$ such that $i=2$, $i=M-1$, $j=2$ or $j=M-1$.
        \item $\mathcal{I}_3$ denotes the set of tuples not in $\mathcal{I}_1$ or $\mathcal{I}_2$.
    \end{itemize}   See \Cref{fig:qs} for a schematic illustration.

    \begin{figure} 
    \centering
        \begin{tikzpicture}[scale=.7]
        \draw[fill=gray] (1,1) rectangle (8,2);
        \draw[fill=gray] (1,1) rectangle (2,8);
        \draw[fill=gray] (1,7) rectangle (8,8);
        \draw[fill=gray] (7,1) rectangle (8,8);
        \draw[xstep=1,ystep=1.0,black,thin] (0,0) grid (9,9);
    \end{tikzpicture}
        \caption{The sets $\mathcal{I}_1$ (outer ring), $\mathcal{I}_2$ (middle ring), $\mathcal{I}_3$ (center).}
    \label{fig:qs}
    \end{figure}
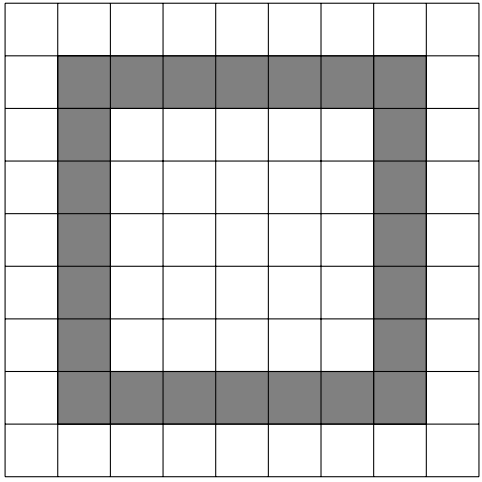

    We define for all $m \in \mathbb{N}$ a preliminary weight $\widetilde{\rho}_m \colon Q_0 \to (0,\infty)$ inductively as follows. Set $\widetilde{\rho}_0 = 1$. Next, assume that the weight $\widetilde{\rho}_{m-1}$ is defined for a given $m \in \mathbb{N}$. Then define $\widetilde{\rho}_{m}$ on the interior of each square $Q(I_m)$ by the formula
    \[\widetilde{\rho}_m(x) = \left\{ \begin{array}{ll} \widetilde{\rho}_{m-1}(x) & \text{ if } (i_m, j_m) \in \mathcal{I}_1 \\
    (M-3) \cdot \widetilde{\rho}_{m-1}(x) & \text{ if } (i_m, j_m) \in \mathcal{I}_2 \\ r \cdot \widetilde{\rho}_{m-1}(x) & \text{ if } (i_m, j_m) \in \mathcal{I}_3 \end{array} \right. .\]
    Extend this definition to all of $Q_0$ by lower semicontinuity. Note that $\widetilde{\rho}_m$ is constant on the interior of each cube $Q(I_m) \in \mathcal{Q}_m$, and so we also denote this value by $\widetilde{\rho}_m(I_m)$.
    
    We now define the weights $\rho_m \colon Q \to (0,\infty)$ in inductive fashion similarly to the $\widetilde{\rho}_m$ but with the following \textit{stopping condition}. In the following, set $J_k(i_k, j_k) = 1$ if $(i_k,j_k) \in \mathcal{I}_1 \cup \mathcal{I}_2$ and $J_k(i_k,j_k) = 0$ if $(i_k, j_k) \in \mathcal{I}_3$. 
    \begin{align*}
      \text{Stopping condition:} & \text{ If } \sum_{k=1}^m J_k(i_k,j_k) > m/3 \text{ for some } m \in \mathbb{N} \text{, then set }
      \rho_{m'}(x) = \rho_{m-1}(x) \\ & \text{ for all } m' \geq m.   
    \end{align*}
    Intuitively, the stopping condition occurs if the weight $\rho_m$ on a cube $Q(I_m)$ has ``stepped up'' or ``stayed flat'' more than a third of the time. Provided $M$ is sufficiently large, the stopping condition occurs only on a small number of cubes. See Section 5 of \cite{Rom:19b} for a justification of this claim for the original construction; in \Cref{prop:f_n} below, we prove the precise version needed for this paper. 
    
    We note that the stopping condition here is slightly different than in \cite{Rom:19b}, which essentially uses the relation $\sum_{k=1}^m J_k(i_k, j_k) > m/2$ as the stopping condition. The new stopping condition allows us to obtain the set $E$ in \Cref{thm:qs} more explicitly than in \cite{Rom:19b}. For each $m \in \mathbb{N}$, let $\widetilde{\mathcal{Q}}_m$ denote the subset of squares in $\mathcal{Q}_m$ that satisfy (or for which one of its parent squares satisfies) the stopping condition. Let $E_m = \bigcup \mathcal{Q}_m \setminus \bigcup \widetilde{\mathcal{Q}}_m$. In words, $E_m$ is the set of points in $Q_0$ which have not been affected by the stopping condition after $m$ levels, and in particular $\rho_m(x) = \widetilde{\rho}_m(x)$ for any such point $x$. Then set $E = \bigcap_{m=1}^\infty E_m$. Observe that $E$ is a compact set.

    For each $m \in \mathbb{N}$, the weight $\rho_m$ induces a length metric on $Q$, denoted by $d_m$. These metrics converge pointwise to a metric $d$; this is shown in Section 3 of \cite{Rom:19b} for the weights $\widetilde{\rho}_m$ but there is no essential difference. We then take $X$ to be the metric space consisting of the set $Q_0$ equipped with the metric $d$. Take $g$ to be the change-of-metric or identity map from $Q_0$ (equipped with the Euclidean metric) onto $X$. 

    \begin{lemm}
        The map $g \colon Q_0 \to X$ is quasisymmetric.
    \end{lemm}

    This is essentially proved as Proposition 4.1 of \cite{Rom:19b}. Note that \cite{Rom:19b} proves the claim for the metric derived from the weights $\widetilde{\rho}_m$ instead of the metric $d$, but again there is no significant difference between these situations. 


    The following is a version of Lemma 3.3 in \cite{Rom:19b}:
    \begin{lemm} \label{lemm:diameter}
        There exists a constant $C_1$ such that for all $Q(I_m) \in \mathcal{Q}_m$, 
        \[\diam g(Q(I_m)) \leq C_1 \rho_m(I_m) M^{-m}. \]
    \end{lemm}

    We now arrive at the main result of this section: that the set $E$ has small Hausdorff dimension in the metric $d$. Here, $\mathcal{H}_{\varepsilon}^D$ denotes the $D$-dimensional Hausdorff $\varepsilon$-content with respect to $d$, and $\mathcal{H}^D$ denotes the $D$-dimensional Hausdorff measure, or Hausdorff $D$-measure. 

    \begin{prop} \label{prop:hausdorff}
    The Hausdorff $1$-measure on $g(E)$ satisfies $\mathcal{H}^{1}(E) = 0$.
    \end{prop}
    \begin{proof}
        Consider the collection $\mathcal{Q}_m$ of cubes of generation $m$. There are $M^{2m}$ such cubes. For a given cube $Q(I_m) \in \mathcal{Q}_m$, the weight $\rho_m(I_m)$ has the form 
        \[\rho_m(I_m) = (M-3)^{a}r^b\]
        for all $x$ in the interior of $Q(I_m)$, where $a+b \leq m$. Observe that
        \[\sum_{k=1}^m J_k(i_k,j_k) = m-b.\]
        
        Thus, if $Q(I_m) \notin \widetilde{\mathcal{Q}}_m$, then $a \leq m-b \leq m/3$. This implies that $b \geq 2m/3$ and hence that $a \leq b/2$. Thus 
        \[\rho_m(I_m) \leq (M-3)^{b/2} \cdot r^b\leq (r^{-1})^{b/2} r^b= r^{b/2} \leq r^{m/3}.\] 

        By \Cref{lemm:diameter}, it follows that $\diam g(Q(I_m)) \leq C_1 r^{m/3}\cdot M^{-m} = C_1(r^{1/3}M^{-1})^m$. 
        Let $\varepsilon_m = C_1(r^{1/3}M^{-1})^m$. We have the upper bound
        \begin{align*}
          \mathcal{H}_{\varepsilon_m}^{1}(g(E_m)) & \leq \sum_{Q(I_m) \in \mathcal{Q}_m} \diam g(Q(I_m)) \\
          & \leq M^{2m} \cdot C_1 (r^{1/3} M^{-1})^m = C_1(r^{1/3} M)^m.
        \end{align*}
        By initially choosing $r$ to satisfy $r < M^{-3}$, we see that $r^{1/3}M <1$. Hence the right-hand side goes to $0$ as $m \to \infty$. Since $E \subset E_m$ for all $m$, this in turn implies that $\mathcal{H}^1(g(E))= 0$. 
    \end{proof}
    
        Note that the choice of dimension for the Hausdorff measure in \Cref{prop:hausdorff} is arbitrary; by adjusting the value of $r$, we get the same conclusion for the Hausdorff $D$-measure for any $D>0$.

    \section{Constructing the functions $f_n$} \label{sec:fn}

    The construction in this section depends on a parameter $n \in \mathbb{N}$ in addition to the $M \in \mathbb{N}$ from the previous section. For each $n \in \mathbb{N}$, we define a function $f_n \colon [0,1] \to [0,1]$ based on modifying a standard Bedford--McMullen carpet with uniform fibers. 

    For the general construction of Bedford--McMullen carpets with uniform fibers, we refer the reader to Sections 1.4 and 4.1 of \cite{BHL:23}, and to the original work of Bedford \cite{Bed:84} and McMullen \cite{McM:84}. To summarize briefly, a \textit{Bedford--McMullen carpet} is a set $A$ constructed in the following way. Choose values $M_2 \geq M_1$, and divide the unit square $Q_0 = [0,1]^2$ into an $M_1 \times M_2$ grid of subrectangles
    \[R(k,l) = \left\{(x,y): \frac{k-1}{M_1} \leq x \leq \frac{k}{M_1}, \frac{l-1}{M_2} \leq y \leq \frac{l}{M_2}\right\}. \]
    Choose one or more subcollections of  $\{R(k,l): 1 \leq k \leq M_1, 1 \leq l \leq M_2\}$, each of which is called a \textit{pattern}. Choose one of these as the initial pattern, and let $A_0$ denote the set of points belonging to the rectangles of this pattern. Define the set $A_1$ by replacing each rectangle comprising $A_0$ with a rescaled and stretched version of the rectangles of the same or a different pattern, according to a predetermined rule. Define the sequence $A_2, A_3, \ldots$ iteratively using the same procedure. Then take $A = \bigcap_{i=1}^\infty A_i$.

    A Bedford--McMullen carpet has \textit{uniform fibers} if each row of each pattern contains the same number of rectangles. While general Beford--McMullen carpets are difficult objects to analyze, the property of uniform fibers guarantees good behavior. In particular, the Hausdorff dimension of such a carpet is well-understood.

    In the remainder of this section, we will consider for each $n \in \mathbb{N}$ a specific Bedford--McMullen carpet that we now describe. Set $M_1 = M^n$ and $M_2 = M$. Divide $Q_0$ into the $M^n \times M$ grid of subrectangles $R(k_1,k_1)$ as above, where $k_1 \in \{1, \ldots, M^n\}$ and $l_1 \in \{1, \ldots, M\}$. Divide each rectangle $R(k_1,l_1)$ in like manner, obtaining rectangles
    $R(k_1, l_1, \ldots, k_m, l_m)$ of size $M^{-nm} \times M^{-m}$, where $k_1, \ldots, k_m \in \{1, \ldots, M^n\}$ and $l_1, \ldots, l_m \in \{1,\ldots, M\}$.



    Consider the sawtooth pattern that is schematically illustrated in \Cref{fig:pattern1}, representing a subset of the rectangles $R(k_1,l_1)$, $k_1,l_1 \in \{1, \ldots, M\}$. This pattern and its horizontal reflection are iterated in such a way that each stage of the construction forms a connected set. Since no two rectangles are contained in the same vertical cross-section, the limit set is the graph of a continuous function from $[0,1]$ to itself, which we denote by $h_n$. Observe that, for each set of indices $k_1,\ldots, k_m$, there is a unique multiindex $l_1,  \ldots, l_m$ such that $R(k_1,l_1, \ldots, k_m, l_m)$ belongs to the $m$-th level of the construction. Denote this rectangle by $S(k_1, \ldots, k_m)$ or by $S(K_m)$, where $K_m$ denotes the multiindex $(k_1, \ldots, k_m)$. Denote by $\mathcal{K}_m$ the set of multiindexes $K_m$ and by $\mathcal{S}_m$ the collection of rectangles $S(K_m)$.

                    \begin{figure} 
    \centering
    \hfill
    \subfloat[Schematic of the pattern used to define $h_n$ ($M=5$ and $n=2$)]{
        \begin{tikzpicture}[scale=1.2]
        \draw[xstep=.2,ystep=1.0,gray,thin] (0,0) grid (5,5);
        \foreach \x in {0,...,4}
         {\filldraw[fill=gray] (.2*\x,\x) rectangle (.2*\x+.2,\x+1);
         \filldraw[fill=gray] (.2*\x+1,5-\x) rectangle (.2*\x+1+.2,4-\x);
         \filldraw[fill=gray] (.2*\x+2,\x) rectangle (.2*\x+2+.2,\x+1);
         \filldraw[fill=gray] (.2*\x+3,5-\x) rectangle (.2*\x+3+.2,4-\x);
         \filldraw[fill=gray] (.2*\x+4,\x) rectangle (.2*\x+4+.2,\x+1);}
        \end{tikzpicture}
    \label{fig:pattern1}}
    \hfill
    \subfloat[Schematic of the pattern used after the stopping condition ($M=5$)]{
        \begin{tikzpicture}[scale=1.2]
        \draw[xstep=1,ystep=1.0,gray,thin] (0,0) grid (5,5);
        \foreach \x in {0,...,4}
         {\filldraw[fill=gray] (\x,\x) rectangle (\x+1,\x+1);}
        \end{tikzpicture}
    \label{fig:pattern2}}
    \hfill \hfill
    \caption{}
    \label{fig:patterns}
    \end{figure}

    Observe that the pattern we use has the property of uniform fibers: each horizontal cross-section of the pattern contains $M^{n-1}$ rectangles. As noted, this property makes the Hausdorff dimension and corresponding Hausdorff measure well-behaved. The following fact is due to McMullen \cite{McM:84}. See also Proposition 4.2 in \cite{BHL:23} for a proof. 

    \begin{lemm} \label{lemm:HausdorffDim}
        The Hausdorff dimension of $\Gamma(h_n)$ is $D_n = 2-1/n$. Moreover, $0 < \mathcal{H}^{D_n}(\Gamma(h_n)) < \infty$. 
    \end{lemm}

    Our next objective is to modify the function $h_n$ by applying the same stopping condition that we used when defining $g$. Note that each rectangle $S(K_m)$ is contained in some rectangle $Q(I_{m})$ of level $m$. Similarly to \Cref{sec:g0}, for each $1 \leq k \leq m$ we set $J_k(K_k) = 1$ if $(i_k,j_k) \in \mathcal{I}_1 \cup \mathcal{I}_2$ and $J_k(k_k) = 0$ if $(i_k,j_k) \in \mathcal{I}_3$. The analogous stopping condition is the following: If $\sum_{k=1}^m J_k(i_k,j_k) > m/3$ for some $m \in \mathbb{N}$, then replace the sawtooth pattern by the simple linear pattern shown in \Cref{fig:pattern2} for all future levels of the construction. Denote the resulting function by $f_n$. 

    Let $\widetilde{\mathcal{S}}_m$ denote the subset of rectangles $S(K_m)$ of level $m$ that satisfies (or one of its parent rectangles satisfies) the stopping condition. Let $B_m$ be the closure of the set $[0,1] \setminus \left( \pi_1 \left( \bigcup \widetilde{\mathcal{S}}_m \right)\right)$. Here, $\pi_1$ denotes projection onto the first coordinate. In words, $B_m$ is the set of points in $[0,1]$ that have not been affected by the stopping condition after $m$ levels. Let $B = \bigcap_{m=1}^\infty B_m$. Observe that $[0,1] \setminus B$ is an open set and that, by construction, the function $f_n$ is locally rectifiable on this set. 

    We now arrive at the most important step of the argument towards proving \Cref{thm:main_a}. This proposition is where we use the requirement that $M \geq 78$. 

    \begin{prop} \label{prop:f_n}
        The set $\Gamma(f_n)$ satisfies $\mathcal{H}^{D_n}(\Gamma(f_n)) \geq \mathcal{H}^{D_n}(\Gamma(h_n))/2>0$.
    \end{prop}

    The remainder of this section is dedicated to proving \Cref{prop:f_n}. This is done by showing that the set on which $\Gamma(f_n)$ and $\Gamma(h_n)$ agree is a set of positive Hausdorff $D_n$-measure. 
    Observe that $h_n(x) = f_n(x)$ for all $x \in B$. Since $h_n(B) \subset \Gamma(f_n)$, it suffices to show that $h_n(B)$ has Hausdorff dimension $D_n$ as well. Observe that, by self-similarity, for each multiindex $K_m$, the set $S(K_m)$ satisfies
    \[\mathcal{H}^{D_n}(\Gamma(h_n) \cap S(K_m)) = \frac{\mathcal{H}^{D_n}(\Gamma(h_n))}{M^{nm}}.\]
    

    For each $m \in \mathbb{N}$, let $G_m$ be the set 
    \[G_m = \bigcup_{i=1}^{M^m} [(i-1+\frac{2}{M})M^{-m},(i-\frac{2}{M})M^{-m}].\] The set $G_m$ is the projection onto each coordinate axis of the union of all center regions of squares in $\mathcal{Q}_m$ (see \Cref{fig:qs}).  
    
    It is convenient to frame the following argument from a probabilistic point of view. Roughly speaking, we interpret the stopping condition as a random walk with barrier and show that with probability at least one-half the barrier is never crossed. We consider the Hausdorff $D_n$-measure on $\Gamma(h_n)$, normalized so that $\Gamma(h_n)$ has measure one, as the underlying probability measure. 
    The argument here is based on that found in Section 6 of \cite{Rom:19b}. Note, however, that a direct adaption of this argument would use the functions $J_k$ as random variables on $\Gamma(h_n)$; however, one encounters a difficulty in our situation because the random variables $J_k$ are not mutually independent and so cannot be used to define the random walk.

    To remedy this difficulty, we consider the two coordinate directions separately. We define random variables $X_m$ and $Y_m$ on the probability space $\Gamma(h_n)$ as follows. Let $\pi_1$ denote projection onto the first coordinate and $\pi_2$ denote projection onto the second coordinate. For a given $x \in \Gamma(h_n)$, there exists a corresponding $K_{m+1} \in \mathcal{K}_{m+1}$ such that $x \in S(K_{m+1})$. Note that finitely many $x$ belong to two such sets, but this does not affect the analysis here.
    
    
    We set $X_m(x) = -1$ if $\pi_1(S(K_{m+1})) \subset G_m$ and $X_m(x) = 1$ otherwise. Likewise, we set $Y_m(x) = -1$ if $\pi_2(S(K_{m+1})) \subset G_m$ and $Y_m(x) = 1$ otherwise. It follows from the construction of $h_n$ that $\mathbb{P}(X_m=-1) = \mathbb{P}(Y_m=-1) = (4-M)/M$ and $\mathbb{P}(X_m=1) = \mathbb{P}(Y_m=1) = 4/M$ for all $m \in \mathbb{N}$. In particular, all these random variables are identically distributed.

    Moreoever, the random variables $X_1, X_2, \ldots$ are mutually independent, as are the variables $Y_1, Y_2, \ldots$. To see this, we observe that for each interval $I_i^m=[(i-1)2^{-m}, i2^{-m}]$, the preimages $\pi_1^{-1}(I_i^m)$ each contain the same number of rectangles from $\mathcal{S}_{m+2}$ lying in $\pi_1^{-1}(G_{m+1})$ and lying outside $\pi_1^{-1}(G_{m+1})$. Likewise, the same is true for the preimages $\pi_2^{-1}(I_i^m)$. 
    

    We consider the partial sums $Z_{N} = \sum_{m=1}^{N} X_m$ as defining a biased random walk that steps down with probability $\mathbb{P}(X_m=-1) = (M-4)/M$ and steps up with probability $\mathbb{P}(X_m=-1) = 4/M$. Let $\Omega$ be the region defined by the condition $Z_N \leq -(3/5)N$. See \Cref{fig:rw}, where the region $\Omega$ is colored blue. Note that if a random walk is contained in $\Omega$, then it steps down at least four times for every time it steps up.  We claim that, given our requirement that $M \geq 78$, with large probability the random walk never leaves the region $\Omega$. 

    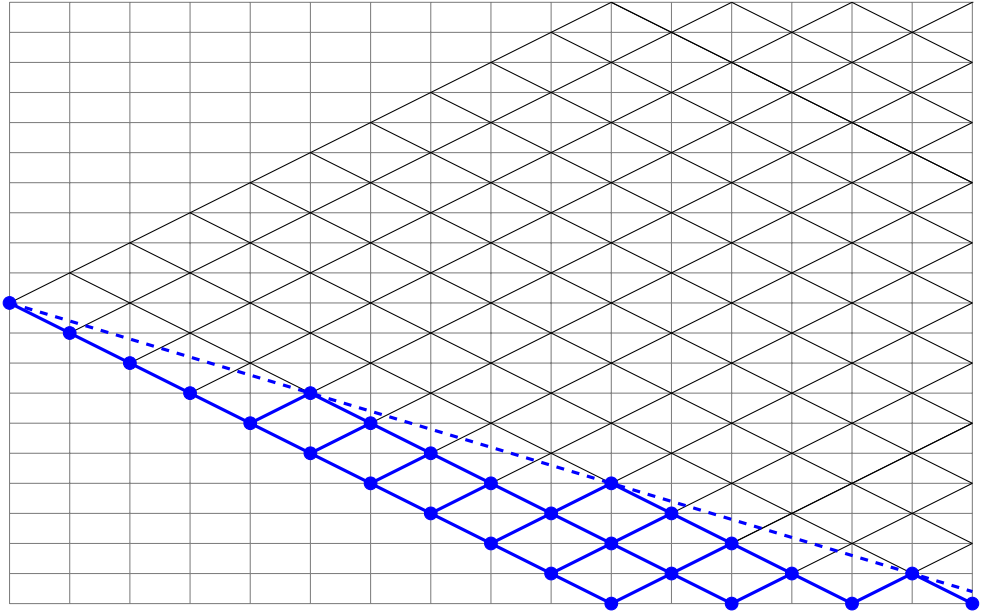
\begin{figure} 
    \centering
        \begin{tikzpicture}[xscale=.8, yscale=.4]
        \draw[xstep=1,ystep=1.0,gray,thin] (0,-10) grid (16,10);
        \foreach \x in {0,...,3}
         {\draw[] (\x,-\x) -- (2*\x+10,10);
         \draw[]  (\x,\x) -- (2*\x+10,-10);
         \draw[] (16-2*\x,10) -- (16,10-2*\x);
         \draw[] (16-2*\x,-10) -- (16,-10+2*\x);}
         \foreach \x in {4,...,10}
         {\draw[] (\x,-\x) -- (16,-2*\x+16);
         \draw[]  (\x,\x) -- (16,2*\x-16);}
         \draw[very thick,blue] (0,0) -- (10,-10);
         \draw[very thick,blue] (4,-4) -- (5,-3) -- (12,-10);\draw[very thick,blue] (5,-5) -- (6,-4);
         \draw[very thick,blue] (6,-6) -- (7,-5);
         \draw[very thick,blue] (7,-7) -- (8,-6);
         \draw[very thick,blue] (8,-8) -- (10,-6) -- (14,-10) -- (15,-9) -- (16,-10);
         \draw[very thick,blue] (9,-9) -- (11,-7);
         \draw[very thick,blue] (10,-10) -- (12,-8);
         \draw[very thick,blue] (12,-10) -- (13,-9);
         \draw[very thick, blue, dashed] (0,0) -- (16,-16*3/5);
         \foreach \x in {0,...,10}
         {\filldraw[blue] (\x,-\x) ellipse (3pt and 6pt);}
        \foreach \x in {5,...,12}
         {\filldraw[blue] (\x,-\x+2) ellipse (3pt and 6pt);}
         \foreach \x in {10,...,14}
         {\filldraw[blue] (\x,-\x+4) ellipse (3pt and 6pt);}
         \filldraw[blue] (15,-9) ellipse (3pt and 6pt);
         \filldraw[blue] (16,-10) ellipse (3pt and 6pt);
    \end{tikzpicture}
        \caption{The random walk and good region $\Omega$ (colored blue).}
    \label{fig:rw}
    \end{figure}

    \begin{lemm}
    Let $p = (M-4)/M$. The probability that the random walk $Z_n$ leaves the region $\Omega$ is at most $(1-p^4)/p^5$.    
    \end{lemm}
    \begin{proof}
        Let $r$ denote the probability that the random walk $Z_n$ leaves in the region $\Omega$. We can bound the value of $r$ as follows. This is a variation on a standard argument for biased random walks with barriers; see, for example, \cite[Section 17.5]{kle:13}.

        The random walk leaves $\Omega$ on the $n$-th step, $n\in \{1,2,3,4\}$, with probability $p^{n-1}(1-p)$. On the fifth step, the random walk may step up, at which point we have returned to the initial configuration. In this case, the probability of leaving $\Omega$ is again $r$. In the other case, if the random walk steps down at step $5$, we may consider the region $\widetilde{\Omega}$ that is the translate of $\Omega$ to the new start point. If the random walk leaves $\widetilde{\Omega}$ (which occurs with probability $r$), then it is again on the boundary of $\Omega$, and the probability of leaving $\Omega$ is at most $r$. This gives the inequality
        \[r \leq (1-p) + p(1-p) + p^2(1-p) + p^3(1-p) + p^4(1-p)r + p^5r^2.\]
        Algebra shows that $r$ must satisfy $r \leq (1-p^4)/p^5$ or $r \geq 1$. However, the case $r \geq 1$ is not possible by the transience of biased random walks. Thus $r$ is at most $(1-p^4)/p^5$. 
    \end{proof}
    Observe that $(1-p^4)/p^5$ goes to zero as $M \to \infty$. More particularly, a calculuation shows that if $M \geq 78$ then ${(1-p^4)/p^5 <.25}$. That is, if $M \geq 45$, then with probability more than $.25$ the random walk satisfies $Z_n \leq -(3/5)n$ for all $n \in \mathbb{N}$. Likewise, with probability more than $.25$ the random walk $W_n = \sum_{n=1}^m Y_m$ satisfies $W_n \leq -(3/5)n$. Thus, with probability more than $.5$ both of these statements are satisfied.

    Next, we observe that for $J_m(I_m)=1$ to hold, then either $X_m(I_m) = 1$ or $Y_m(I_m)=1$ (or both) must hold. If $Z_n$ is contained in $\Omega$, it must step down at least four times for each step up, and likewise for $W_n$. For random walks $Z_n$ and $W_n$ contained in $\Omega$, the two walks together cannot step up more than twice for each four steps down. Thus, the corresponding $J_m$ must satisfy $\sum_{k=1}^m J_k(I_k) \leq m/3$ for all $m$. We conclude that with probability at least $.5$ the stopping condition never occurs. Translating this back in terms of the set $\Gamma(h_n)$, we see that $h_n(B)$, and hence $f_n(B)$, has Hausdorff $D_n$-measure at least $\mathcal{H}^{D_n}(\Gamma(h_n))/2$. This completes the proof of \Cref{prop:f_n}.

    \section{Proof of the main theorems} \label{sec:proof}

    \subsection{Proof of \Cref{thm:main_a}}

    We let $\alpha \in (1,2)$ and pick $n \in \mathbb{N}$ so that $D_n = 2-1/n > \alpha$. Take $f$ to be the function $f_n$, $g$ to be the restriction of $g_0$ to $\Gamma(f_n)$, and $Y = g(\Gamma(f_n))$. From \Cref{prop:f_n}, we see that $\Gamma(f_n)$ has Hausdorff dimension $D_n>\alpha$. From \Cref{prop:hausdorff}, it follows that $g(\Gamma(f) \cap E)$ has Hausdorff dimension at most $1$. Since $\Gamma(f)$ is locally rectifiable outside the closed set $E$ and $g$ is Lipschitz, it also follows that $g(\Gamma(f) \setminus E)$ has Hausdorff dimension at most $1$. The finally claim to verify is the statement about $\dim_H (\Gamma(f_n) \cap (\mathbb{R} \times \{a\}))$. 

    Observe first that the original map $h_n$ satisfies $\dim_H (\Gamma(h_n) \cap (\mathbb{R} \times \{a\})) = D_n-1$ for almost every $a \in [0,1]$; see Proposition 4.2 in \cite{BHL:23} for a proof. In fact, for almost every $a \in [0,1]$ the set $\Gamma(h_n) \cap (\mathbb{R} \times \{a\})$ has positive and finite Hausdorff $(D_n-1)$-measure. Let $A'$ denote the set of such points $a$.

    Recall that $\mathcal{S}_m$ is the collection of rectangles of level $m$ used to define $\Gamma(h_n)$. Let $\mathcal{T}_m$ be the subcollection of $\mathcal{S}_m$ consisting of those rectangles not affected by the stopping condition. Observe that $|\mathcal{T}_m|$, the cardinality of $\mathcal{T}_m$, is at least half $|\mathcal{S}_m|$ for all $m$. 

    Next, for all $j \in \{1, \ldots, M^m\}$, let $\mathcal{S}_m^j$ denote the subcollection of $\mathcal{S}_m$ of rectangles contained in the strip $\mathbb{R} \times [(j-1)M^{-m}, j M^{-m}]$. Note that $|\mathcal{S}_m| = M^{mn}$ and $|\mathcal{S}_m^j| = M^{mn}/M^m = M^{m(n-1)}$. Define the collections $\mathcal{T}_m^j$ similarly. Then
    \[M^{mn} = |\mathcal{S}_m| \leq 2 |\mathcal{T}_m| = 2 \sum_{j=1}^{M^m}|\mathcal{T}_m^j|.\]
    Since $|\mathcal{T}_m^j| \leq |\mathcal{S}_m^j| = M^{m(n-1)}$ for all $j$, it follows that 
    \[|\mathcal{T}_m^j| \geq 4^{-1} |\mathcal{S}_m^j| = 4^{-1} M^{m(n-1)}\]
    for at least $M^m/4$ of the indices $j \in \{1, \ldots, M^m\}$. Otherwise, we would obtain the contradiction 
    \[\sum_{j=1}^{M^m} |\mathcal{T}_m^j| \leq \frac{3M^m}{4} \cdot \frac{M^{m(n-1)}}{4}  + \frac{M^m}{4} \cdot M^{m(n-1)} = \frac{7M^{mn}}{16}.\]
    Let $A_m$ denote the union of those intervals $[(j-1)M^{-m}, jM^{-m}]$ for which $|\mathcal{T}_m^j| \geq 4^{-1}M^{m(n-1)}$. Then the Lebesgue measure of $A_m$ is at least $1/4$. Moreover, we observe that $A_{m+1} \subset A_m$ for all $m \in \mathbb{N}$. 
    Let $A$ denote the set $\bigcap_{m=1}^\infty A_m$ with all points of the form $jM^{-m}$ removed. It is immediate from the outer regularity of Lebesgue measure that $A$ has Lebesgue measure at least $1/4$. 
    
    We claim that for each $a \in A \cap A'$, the set $\Gamma(f_n) \cap (\mathbb{R} \times \{a\})$ has positive Hausdorff $(D_n-1)$-measure. 
    Note that, unlike in the proof of \Cref{prop:f_n}, we cannot rely on self-similarity to prove this claim. Instead, we apply the standard Frostman Lemma (see for example Chapter 8 in \cite{Mat:95}) as follows. Let $\Gamma_a = \Gamma(f_n) \cap (\mathbb{R} \times \{a\})$ and let $\widetilde{\Gamma}_a = \Gamma(h_n) \cap (\mathbb{R} \times \{a\})$. Define a Borel regular measure $\mu$ on $\widetilde{\Gamma}_a$ in the standard way by redistributing a unit mass equally on the rectangles in $\mathcal{S}_m^{j_m}$, where $j_m$ is such that $a \in [(j_m-1)M^{m}, j_mM^{m}]$. More precisely, we specify $\mu$ by declaring that $\mu(S \cap \widetilde{\Gamma}_a) = M^{-m(n-1)}$ for each $S \in S_m^{j_m}$. Note that each rectangle in $\mathcal{S}_m^{j_m}$ contains exactly $M^{n-1}$ rectangles in $\mathcal{S}_{m+1}^{j_{m+1}}$ that intersect $\widetilde{\Gamma}_a$, which justifies that $\mu$ is well-defined as a measure. One can check that the measure $\mu$ satisfies $\mu(B(x,r)) \leq r^{D_n-1}$ for all $x \in \mathbb{R}^2$ and $r>0$. 
    
    The restriction of $\mu$ to $\Gamma_a$, which we denote by $\sigma$, also satisfies $\sigma(B(x,r)) \leq r^{D_n-1}$ for all $x \in \mathbb{R}^2$ and $r>0$. Now $\Gamma_a = \bigcap_{m=1}^\infty \mathcal{T}_m^{j_m}$. Since $|\mathcal{T}_m^{j_m}| \geq 4^{-1}|\mathcal{S}_m^{j_m}|$, it follows from the outer regularity of $\mu$ that $\sigma(\Gamma_a) \geq 4^{-1}\mu(\widetilde{\Gamma}_a)>0$.  
    By the Frostman Lemma, these properties suffice to show that $\Gamma_a$ satisfies $\mathcal{H}^{D_n-1}(\Gamma_a)>0$. Since $\Gamma_a \subset \widetilde{\Gamma}_a$, it follows that $\Gamma_a$ has Hausdorff dimension exactly $D_n-1$, with positive Hausdorff $(D_n-1)$-measure. 
    
    \subsection{Proof of \Cref{thm:main}} We are now ready to define the function $f$ in \Cref{thm:main}. This is done by pasting together rescaled versions of each $f_n$ on dyadic intervals. On each interval $[2^{-n}, 2^{-n+1}]$, define $f$ by
    \[f(x) = 2^{-n} + 2^{-n}f_n(2^n(x-2^{-n})).\]
    We also set $f(0) = 0$. Observe that $f(2^{-n}) = 2^{-n}$ and in particular the function $f(x)$ is well defined and continuous. 

    For each $n \in \mathbb{N}$, the set $f([2^{-n}, 2^{-n+1}])$ has Hausdorff dimension $2-1/n$ by \Cref{lemm:HausdorffDim}. Thus $\Gamma(f)$ has Hausdorff dimension $2$.  

    We now define the metric space $Y$ as follows. Let $X_n$ denote a copy of the space $2^{-n}X$. Take $\widetilde{Y}$ to be the metric completion of the gluing of the spaces $X_n$ over all $n \in \mathbb{N}$ in which the bottom left corner of $X_n$ is identified with the top right corner of $X_{n+1}$. We give $\widetilde{Y}$ the usual length metric induced by the gluing. As a set, we identify $\widetilde{Y}$ with the closure of the set $\bigcup_{n=1}^\infty [2^{-n}, 2^{-n+1}] \times [2^{-n}, 2^{-n+1}]$ in $\mathbb{R}^2$. With this identification, we define $g$ to be the restriction to $\Gamma(f)$ of the identity map from $\bigcup_{n=1}^\infty [2^{-n}, 2^{-n+1}] \times [2^{-n}, 2^{-n+1}]$ to $Y$. Let $E_n = 2^{-n}E + (2^{-n},2^{-n})$. Let $\widetilde{E} = \bigcup_{n=1}^\infty E_n \subset [0,1]$. Let $Y = g(\Gamma(f))$. 

    

    Let $\widetilde{B}_n = 2^{-n}B+2^{-n}$. Then $\widetilde{B} = \bigcup_{n=1}^\infty \widetilde{B}_n$ is the set of points $x$ whose image $f(x)$ is not affected by the stopping condition. Observe that $[0,1] \setminus \widetilde{B}$ is a countable union of open intervals. On each such interval $f$ is Lipschitz, and in particular $f([0,1] \setminus \widetilde{B})$ is locally rectifiable. By construction, $f(\widetilde{B})$ is contained in $\widetilde{E}$, whose image in $Y$ has Hausdorff dimension at most $1$. It follows that $\Gamma(f)$ is the union of two subsets of Hausdorff dimension at most $1$ and thus has Hausdorff dimension $1$. This concludes the proof of \Cref{thm:main}. 
	
	\bibliographystyle{abbrv}  
	\bibliography{biblio}

\end{document}